\newcommand{\si}[1]{#1}
\newcommand{\jo}[1]{}

\jo{
\documentclass[smallextended,referee,envcountsect]{svjour3} 
}
\jo{
\smartqed  
}
\jo{
\usepackage{graphicx}
\usepackage{amssymb,latexsym,amsmath}
\newcommand{\R}{{\mathbb R}}
\newcommand{\T}{{\mathtt T}}
 \journalname{JOTA}
}
\jo{
\begin{document}
\title{An Extension of Yuan's Lemma and its Applications in Optimization
}


\author{G. Haeser
}


\institute{G. Haeser \at
              Department of Applied Mathematics, University of S\~ao Paulo, S\~ao Paulo SP, Brazil. This research was conducted while holding a Visiting Scholar position at Department of Management Science and Engineering, Stanford University, Stanford CA, USA. \\
              \email{ghaeser@ime.usp.br}           
}

\date{Received: date / Accepted: date}
}

\si{
\documentclass{article}
\usepackage{amssymb,latexsym,amsthm,amsmath}
\usepackage{verbatim}
\usepackage{hyperref}

\theoremstyle{plain}
\newtheorem{theorem}{Theorem}
\newtheorem{corollary}{Corollary}
\newtheorem*{main}{Main~Theorem}
\newtheorem{lemma}{Lemma}
\newtheorem*{conjecture}{Conjecture}
\newtheorem{proposition}{Proposition}
\theoremstyle{definition}
\newtheorem{definition}{Definition}
\newtheorem{example}{Example}
\newtheorem{counter}{Counter-Example}
\newtheorem{assumption}{Assumption}

\theoremstyle{remark}

\newtheorem{remark}{Remark}

\numberwithin{equation}{section}

\newcommand{\R}{{\mathbb R}}
\newcommand{\T}{{\mathtt T}}

\title{An extension of Yuan's Lemma and its applications in optimization}

\author{G. Haeser\thanks{Department of Applied Mathematics,
    University of S\~ao Paulo, S\~ao Paulo, SP, Brazil.
    Email: {\tt ghaeser@ime.usp.br}. This research was conducted while holding a Visiting Scholar position at Department of Management Science and Engineering, Stanford University, Stanford CA, USA.}}
\date{March 1, 2017. Reviewed April 7, 2017.}
\begin{document}
}

\maketitle

\begin{abstract}

We prove an extension of Yuan's Lemma to more than two matrices, as long as the set of matrices has rank at most $2$. This is used to generalize the main result of {\it [A. Baccari and A. Trad. On the classical necessary second-order optimality conditions in the presence of equality and inequality constraints. SIAM J. Opt., 15(2):394--408, 2005]}, where the classical necessary second-order optimality condition is proved under the assumption that the set of Lagrange multipliers is a bounded line segment. We prove the result under the more general assumption that the Hessian of the Lagrangian evaluated at the vertices of the Lagrange multiplier set is a matrix set with at most rank $2$.
 We apply the results to prove the classical second-order optimality condition to problems with quadratic constraints and without constant rank of the Jacobian matrix.

\jo{
}
\si{{\bf Keywords: }Quadratic forms, second-order optimality conditions, global convergence.
}
\end{abstract}

\section{Introduction}

In a general nonlinear optimization problem, it is well known that when the set of Lagrange multipliers is bounded, a local minimizer satisfies a second-order optimality condition in the form of the maximum of many quadratic forms being non-negative in the critical cone, where each quadratic form is defined by the Hessian of the Lagrangian function evaluated at different Lagrange multipliers.

In this paper we are interested in second-order optimality conditions that can be verified with one single Lagrange multiplier. This is motivated by algorithmic considerations, since algorithms usually generate a primal-dual sequence and one is interested in proving that limit points of this sequence satisfy a first- or second-order optimality condition. Under this setting, one is usually restricted to considering a smaller subset of the true critical cone.

Yuan's Lemma \cite{yuan} gives an important tool for this type of results, since it states that when the maximum of two quadratic forms is non-negative, some convex combination of the matrices is positive semidefinite. This implies that when the set of Lagrange multipliers is a bounded line segment, one can find a Lagrange multiplier such that the Hessian of the Lagrangian is positive semidefinite on a subset of the critical cone. This is the main result of \cite{baccaritrad}. However, this approach does not work when there are more than two quadratic forms, that is, when the Lagrange multiplier set is larger than a line segment.

In Section 2, we will extend Yuan's Lemma to more than two quadratic forms. This will be done under an assumption of redundancy on the matrices defining the quadratic forms, in the sense that it is a matrix set of rank at most $2$. This gives, in Section 3, an optimality condition verifiable on one single Lagrange multiplier under an assumption of redundancy of the set of Hessians of the Lagrangian matrices that includes as a particular case the case of a Lagrange multiplier set equals to a line segment. In Section 4 we apply the results to a special form of quadratically-constrained problems, and we prove that our assumption holds when the rank of the Jacobian matrix increases at most by one in the neighborhood of a local solution. Section 5 gives some conclusions.

{\bf Notation:} A set $K\subseteq\R^n$ is a first-order cone if $K$ is the direct sum of a subspace and a ray, where a ray is a set of the form $\{td_0\mid t\geq0\}$, for some $d_0\in\R^n$. Given a set $K\subseteq\R^n$ and a symmetric matrix $A\in\R^{n\times n}$, we say that $A$ is positive semidefinite on $K$ if $x^\T Ax\geq0$ for all $x\in K$. We denote by $\R^m_+$ the non-negative orthant of $\R^m$ and $\Lambda_m=\{t\in\R^m_+\mid \sum_{i=1}^mt_i=1\}$ the $m$-dimensional simplex. The rank of a set of matrices is the maximum number of linearly independent matrices viewed as vectors on the appropriate space.

\section{Extension of Yuan's Lemma}


An important result related to the trust-region subproblem is the following:

\begin{lemma}[Yuan's Lemma \cite{yuan}] 
\label{yuan}
Let $A,B\in\R^{n\times n}$ be symmetric matrices and $K\subseteq\R^n$ be a first-order cone. Then the following conditions are equivalent:
\begin{itemize}
\item $\max\{d^\mathtt{T}Ad,d^\mathtt{T}B d\}\geq0, \forall d\in K$,
\item there exists $t_1\geq0, t_2\geq0$ with $t_1+t_2=1$ such that $t_1A+t_2B$ is positive semidefinite on $K$.
\end{itemize}
\end{lemma}

 
 Lemma \ref{yuan} is known as Yuan's Lemma \cite{yuan} and it has had many important generalizations \cite{baccaritrad,hiriart,seeger1,seeger2,noteyuan}. In particular, our presentation of Yuan's Lemma corresponds to \cite[Corollary 3.2]{baccaritrad}. Yuan's Lemma when $K=\R^n$ can be easily proved with a separation argument due to the convexity of the set $\{(x^\T Ax, x^\T Bx)\mid x\in\R^n\}\subseteq\R^2$ given in \cite{dines}. See also some related convexity results in \cite{brickman,polyak98,polyak01,surveySlemma,xia2014}.
  
%

In the next result we extend Yuan's Lemma to $m$ matrices $A_1,\dots,A_m\in\R^{n\times n}$ with $\mbox{rank}(\{A_1,\dots,A_m\})\leq2$. A similar assumption is used to obtain a result of S-Lemma type in \cite[Proposition 3.5]{surveySlemma}, however, our proof is elementary and our result is more general in the sense that it holds for first-order cones.
\begin{lemma}
\label{yuanrank}
Let $A_i\in\R^{n\times n}, i=1,\dots,m, m\geq2$, be such that $A_i=\alpha_i A_1+\beta_i A_2, i=3,\dots,m$ for some $(\alpha_i,\beta_i)\in\R^2, i=3,\dots,m$ and $K\subseteq\R^n$ a first-order cone. Then, the following are equivalent:
\begin{equation}
\label{max-m-quad}
\max_{i=1,\dots,m}\{x^\T A_ix\}\geq0,\forall x\in K,
\end{equation}
\begin{equation}
\label{posit}\exists t\in\Lambda_m\mbox{ such that } \sum_{i=1}^mt_iA_i\mbox{ is positive semidefinite on }K.
\end{equation}
\end{lemma} 
\begin{proof}
The fact that \eqref{posit} implies \eqref{max-m-quad} can be easily seen by contradiction. The reciprocal implication is done by induction on $m$. If $m=2$, the result follows from Lemma \ref{yuan}. Assume the assertion is true for $m-1\geq2$. Assume \eqref{max-m-quad}. Since $A_m=\alpha_mA_1+\beta_mA_2$, let us consider the following cases:\\
The case $(\alpha_m,\beta_m)=(0,0)$ is trivial. If $\alpha_m\geq0$ and $\beta_m\geq0$, $(\alpha_m,\beta_m)\neq(0,0)$, then $x^\T A_1x<0$ and $x^\T A_2x<0$ imply $x^\T A_mx<0$. Thus, \eqref{max-m-quad} implies $\max_{i=1,\dots,m-1}\{x^\T A_ix\}\geq0,\forall x\in K$. If $\alpha_m<0$ and $\beta_m>0$, we have $x^\T A_1x<0$ and $x^\T A_mx<0$ imply $x^\T A_2x<-\frac{\alpha_m}{\beta_m}x^\T A_1x<0$, hence, \eqref{max-m-quad} implies $\max_{i=1,\dots,m,i\neq2}\{x^\T A_ix\}\geq0,\forall x\in K$. It is clear that any matrix $A_i, i=1,\dots,m,i\neq2$ can be written as the linear combination of two fixed matrices in this set. The case $\alpha_m<0$ and $\beta_m=0$ gives $\max\{x^\T A_1x,x^\T A_mx\}\geq0,\forall x\in K$. Also, the case $\alpha\geq0, \beta<0$ is analogous. In any of the above cases, the result follows from Lemma \ref{yuan} or the inductive assumption.
If $\alpha_m<0$ and $\beta_m<0$, we have $t_1A_1+t_2A_2+t_mA_m=(t_1+\alpha_m t_m)A_1+(t_2+\beta_m t_m)A_2=0$, which is positive semidefinite on $K$, for $(t_1,t_2,t_m)=(\frac{-\alpha_m}{1-\alpha_m-\beta_m},\frac{-\beta_m}{1-\alpha_m-\beta_m},\frac{1}{1-\alpha_m-\beta_m})\in\Lambda_3$.
\jo{\qed}
\end{proof}

\begin{example}
\label{firstexample}
Let $$A_1=\left(\begin{array}{cc}1&-1\\-1&1\end{array}\right), A_2=\left(\begin{array}{cc}-2&1\\1&1\end{array}\right), A_3=\left(\begin{array}{cc}4&-3\\-3&1\end{array}\right).$$
Note that $A_3=2A_1-A_2$. A simple calculation shows that $\max_{i=1,2,3}\{x^\T A_ix\}\geq0$ for all $x$, and, as predicted by Lemma \ref{yuanrank}, $t_1A_1+t_2A_2+t_3A_3$ is positive semidefinite, for instance, for $(t_1,t_2,t_3)=(0,\frac{3}{5},\frac{2}{5})$. As it is done in the proof, this is a consequence of the condition $\max\{x^\T A_2x,x^\T A_3x\}\geq0$ for all $x$ that necessarily holds. 
\end{example}

\begin{example}
\label{secondexample}
Let $$A_1=\left(\begin{array}{cc}-1&0\\0&1\end{array}\right), A_2=\left(\begin{array}{cc}1&2\\2&-2\end{array}\right), A_3=\left(\begin{array}{cc}0&-2\\-2&1\end{array}\right).$$
Note that $A_3=-A_1-A_2$. Therefore, $\frac{1}{3}A_1+\frac{1}{3}A_2+\frac{1}3A_3=0$ is positive semidefinite and clearly it can not exist $x\in\R^2$ with $x^\T A_1x<0$, $x^\T A_2x<0$ and $x^\T A_3x<0$, which implies the maximum of the three quadratic forms is non-negative. Note, however, that for all $i\neq j$, it is not the case that $\max\{x^\T A_ix,x^\T A_jx\}\geq0$ for all $x$.
\end{example}

\section{A Second-order Optimality Condition}

Let us consider the nonlinear optimization problem

\begin{equation}
\label{genp}
\begin{array}{ll}\mbox{Minimize}&f(x),\\
\mbox{subject to}&h(x)=0,\\& g(x)\leq 0,\end{array}\end{equation}

where $f:\R^n\to\R, h:\R^n\to\R^{p_1}, g:\R^n\to\R^{p_2}$ are twice continuously differentiable functions. For a feasible $x$, we define $A(x)=\{i\in\{1,\dots,p_2\}\mid g_i(x)=0\}$, the index set of active inequality constraints.  Given $(\lambda,\mu)\in\R^{p_1}\times\R^{p_2}_+$, we consider the Lagrangian function
$$x\mapsto L(x,\lambda,\mu)=f(x)+\sum_{i=1}^{p_1}\lambda_ih_i(x)+\sum_{i=1}^{p_2}\mu_ig_i(x),$$ with gradient $\nabla L(x,\lambda,\mu)$ and Hessian $\nabla^2 L(x,\lambda,\mu)$, where derivative is taken with respect to $x$. We will assume that at a local minimizer $x^*$ of \eqref{genp}, the Mangasarian-Fromovitz constraint qualification holds:

\begin{definition}
A feasible point $x^*$ of problem \eqref{genp} satisfies the Mangasarian-Fromovitz constraint qualification \cite{mfcq} when $(\alpha,\beta)=(0,0)$ is the only solution of $$\sum_{i=1}^{p_1}\alpha_i\nabla h_i(x^*)+\sum_{i\in A(x^*)}\beta_i\nabla g_i(x^*)=0, \alpha_i\in\R, i=1,\dots,p_1; \beta_i\geq0, i\in A(x^*).$$
\end{definition}

It is well known \cite{gauvin} that at a local minimizer $x^*$, the Mangasarian-Fromovitz constraint qualification is equivalent to the non-emptyness and boundedness of the Lagrange multiplier set $$\Lambda(x^*)=\{(\lambda,\mu)\in\R^{p_1}\times\R^{p_2}_+\mid \nabla L(x^*,\lambda,\mu)=0, \mu^\T g(x^*)=0\}.$$ We observe also that $\Lambda(x^*)$ is always a closed polyhedron. The following optimality condition is well-known:

\begin{theorem}[\cite{bshapiro}] \label{teo-bs}
Let $x^*$ be a local minimizer of \eqref{genp} satisfying the Mangasarian-Fromovitz constraint qualification. Then, for each $d\in C(x^*)$, there is a Lagrange multiplier $(\lambda,\mu)\in\Lambda(x^*)$ satisfying
\begin{equation}
\label{quadform}d^\T\nabla^2 L(x^*,\lambda,\mu)d\geq0,
\end{equation}
where \begin{equation}
\label{critical-cone}
C(x^*)=\left\{\begin{array}{ll}d\in\R^n:& \nabla h_i(x^*)^\T d=0, i=1,\dots,p_1; \nabla g_i(x^*)^\T d\leq0, i\in A(x^*);\\
& \nabla f(x^*)^\T d=0.\end{array}\right\}
\end{equation}
is the critical cone.
\end{theorem}
 
Theorem \ref{teo-bs} is the classical second-order necessary optimality condition under Mangasarian-Fromovitz constraint qualification alone. A version of it without constraint qualifications and using Fritz-John multipliers is known as a no-gap optimality condition in terms of sufficiency, see details in \cite{bshapiro}.

Note that condition \eqref{quadform} holds with different Lagrange multipliers for different directions in the critical cone. In this paper we are interested in conditions that ensure the validity of \eqref{quadform} for the same Lagrange multiplier $(\lambda,\mu)\in\Lambda(x^*)$, even if for that we need to consider a smaller subset of the true critical cone $C(x^*)$. This is motivated by algorithmic considerations, since algorithms usually generate only one Lagrange multiplier approximation. See the discussion in \cite{conjnino}.

In the next theorem, we will use Lemma \ref{yuanrank} to formulate a new optimality condition of this type. It is a generalization of the main result of \cite{baccaritrad}, where $\Lambda(x^*)$ was assumed to be a bounded line segment.

\begin{theorem}
\label{main}
Let $x^*$ be a local minimizer of \eqref{genp} satisfying the Mangasarian-Fromovitz constraint qualification. Let $(\lambda^1,\mu^1),\dots,(\lambda^v,\mu^v)$ be the vertices of the Lagrange multiplier set $\Lambda(x^*)$. If the rank of $\{\nabla^2 L(x^*,\lambda^i,\mu^i)\}_{i=1}^v$ is at most $2$, then for every first-order cone $K\subseteq C(x^*)$, there exists $(\lambda,\mu)\in\Lambda(x^*)$ such that
\begin{equation}
\label{coneK}d^\T\nabla^2L(x^*,\lambda,\mu)d\geq0, \forall d\in K.
\end{equation}
\end{theorem}

\begin{proof}
Let $K\subseteq C(x^*)$ be a first-order cone. From Theorem \ref{teo-bs}, we have 
\begin{equation}
\label{maxim}
\max_{(\lambda,\mu)\in\Lambda(x^*)}\{d^\mathtt{T}\nabla^2 L(x^*,\lambda,\mu)d\}\geq0, \forall d\in K.
\end{equation}
Since Mangasarian-Fromovitz constraint qualification implies that $\Lambda(x^*)$ is a non-empy, compact, polyhedral set, let $(\lambda^1,\mu^1),\dots,(\lambda^v,\mu^v)$ be its vertices.
Since for each $d\in K$ the maximization problem \eqref{maxim} is a linear programming with non-empty and compact feasible region, its solution is attained at a vertex, which implies
\begin{equation}
\label{maxim-finite}
\max_{i=1,\dots,v}\{d^\mathtt{T}\nabla^2 L(x^*,\lambda^i,\mu^i)d\}\geq0, \forall d\in K.
\end{equation}
Lemma \ref{yuanrank} gives $t\in\Lambda_v$ such that $\sum_{i=1}^vt_i\nabla^2 L(x^*,\lambda^i,\mu^i)$ is positive semidefinite on $K$, and the linearity of $\nabla^2 L(x^*,\cdot,\cdot)$ gives $\nabla^2 L(x^*,\lambda,\mu)$ positive semidefinite on $K$ for $(\lambda,\mu)=\sum_{i=1}^vt_i(\lambda^i,\mu^i)\in\Lambda(x^*)$.
\jo{\qed}
\end{proof}

Note that when the critical cone is a first-order cone, Theorem \ref{main} gives a full second-order optimality condition in terms of the true critical cone. One condition ensuring this, is the generalized scrict complementarity slackness, namely, that there is at most one index $i_0\in A(x^*)$ such that $\mu_{i_0}=0$ for all $(\lambda,\mu)\in\Lambda(x^*)$. See \cite{baccaritrad}.

Let us revisit Examples \ref{firstexample} and \ref{secondexample}.
The fact that $\max_{i=1,2,3}\{x^\T A_ix\}\geq0$ for all $x$, implies that $x^*=(0,0), z^*=0$ is a global minimizer of the problem of minimizing $z$, subject to $\frac{1}{2}x^\T A_ix-z\leq0, i=1,2,3$. It is easy to see that the Mangasarian-Fromovitz constraint qualification holds, the set of Lagrange multipliers at $(x^*,z^*)$ is the simplex $\Lambda_3$ and the critical cone is $\R^2\times\{0\}$. Note that the quadratic form defined by the Hessian of the Lagrangian at $(x^*,z^*)$ evaluated at each of the three vertices of $\Lambda_3$ and restricted to the critical cone gives the quadratic forms defined by the three matrices $A_1, A_2$ and $A_3$. Hence, the optimality condition of Theorem \ref{main} translates precisely to the fact that $t_1A_1+t_2A_2+t_3A_3$ is positive semidefinite for some $(t_1,t_2,t_3)\in\Lambda_3$, which holds as observed in Examples \ref{firstexample} and \ref{secondexample}. 
In the next session we will give a sufficient condition for the rank assumption to hold under this setting.

\section{Application to a Quadratically-constrained Problem}

It is well known that without the rank assumption, the optimality condition of Theorem \ref{main}  does not hold. Known counter-examples \cite{aru,ani,Baccari2004} are of the form
\begin{equation}
\label{quad}
\begin{array}{ll}\mbox{Minimize}&z,\\
\mbox{subject to}&g_i(x,z):=\frac{1}{2}x^\T A_ix-z\leq 0,i=1,\dots,m,\end{array}\end{equation}
where $A_i\in\R^{n\times n}, i=1,\dots,m$, are symmetric matrices and $(x^*,z^*)=(0,0)\in\R^n\times\R$ is a solution. This is a special case of the well known quadratically-constrained quadratic programming problem, a well-studied difficult non-convex optimization problem \cite{qcqp,pardalos,more}. Note that the linear objective can replace an objective function $f(x)$ by adding the constraint $f(x)-z\leq0$. Note that the critical cone $C(x^*,z^*)$ is the subspace $\R^n\times\{0\}$.
We will prove that for this type of problem, in order to fulfill our rank assumption, it is sufficient to assume that the rank of the Jacobian matrix (that is, the matrix of gradients of active constraints at the solution) increases at most by one in a neighborhood of the solution. It is conjectured in \cite{ams2} that for the general problem \eqref{genp} satisfying Mangasarian-Fromovitz constraint qualification at the local solution $x^*$, when the rank of the Jacobian matrix increases at most by one in a neighborhood of the solution, the optimality condition \eqref{coneK} holds with the cone $K$ equal to the lineality space of $C(x^*)$. This is known to hold when the rank of the Jacobian matrix is constant \cite{ams2}, one less than the number of active constraints \cite{baccaritrad}, one less than the dimension $n$ \cite{conjnino}, or when the Jacobian matrix has a smooth singular value decomposition \cite{conjnino}. We refer the reader to \cite{conjnino} for a discussion about this conjecture and related results. Our result proves this conjecture for the class of quadratically-constrained problems \eqref{quad}. Note that this is not a particular case of any known case where the conjecture holds. In particular, \cite[Example 3.3]{conjnino} shows a problem  of type \eqref{quad} with $m=3$ active constraints and $n=3$ variables where the singular value decomposition is not smooth, and the rank of the Jacobian matrix increases from $1$ to at most $2$ in a neighborhood of the origin. Our theorem applies to this problem, whereas no other known result applies to obtain \eqref{coneK}.
To present our result, we start with the following lemma:

\begin{lemma}
\label{jacobrank}
Given $a\neq0$ and $A_i\in\R^{n\times n}$ symmetric matrices, $i=1,\dots,m$, define $v_i(x)=\left(\begin{array}{c}A_ix\\a\end{array}\right)\in\R^{n+1}, i=1,\dots,m$ and $J(x)\in\R^{(n+1)\times m}$ given column-wise by $J(x)=[v_1(x) \cdots v_m(x)]$. If the rank of $J(x)$ is at most $2$ for all $x\in\R^n$, then the rank of $\{A_1,\dots,A_m\}$ is at most $2$.
\end{lemma}

\begin{proof}
Take any three distinct indexes $\{i_1,i_2,i_3\}\subseteq\{1,\dots,m\}$ and let $A=A_{i_1}, B=A_{i_2}, C=A_{i_3}$. Let us prove that $\{A,B,C\}$ is linearly dependent. Since the submatrix of $J(x)$ consisting of columns $i_1, i_2$ and $i_3$ has rank at most $2$, we have that for all $x\in\R^n$, there is $(\alpha_x,\beta_x,\gamma_x)\neq(0,0,0)$ such that $\alpha_xAx+\beta_xBx+\gamma_xCx=0$ and $\alpha_x+\beta_x+\gamma_x=0$. That is, \begin{equation}\label{dep}\forall x\in\R^n,\alpha_x(A-C)x+\beta_x(B-C)x=0, \mbox{ for some }(\alpha_x,\beta_x)\neq(0,0).\end{equation}
Let $v_1,\dots,v_n\in\R^n$ be an orthonormal basis of eigenvectors of the symmetric matrix $B-C$ with corresponding eigenvalues $\lambda_1,\dots,\lambda_n\in\R$. If all eigenvalues are zero, then $B=C$ and the result follows. Suppose $\lambda_1\neq0$ and let us apply \eqref{dep} with $x=v_1$:
$$\alpha_1(A-C)v_1+\beta_1\lambda_1v_1=0,\mbox{ for some }(\alpha_1,\beta_1)\neq(0,0).$$
For $i=1,\dots,n$ let us take the inner product with $v_i$. We have: 
$$\alpha_1 v_1^\T(A-C)v_1+\beta_1\lambda_1=0\mbox{ and }\alpha_1v_i^\T(A-C)v_1=0, i>1.$$
Since $(\alpha_1,\beta_1)\neq(0,0)$ and $\lambda_1\neq0$, we have $\alpha_1\neq0$. That is, \begin{equation}\label{aux1}v_1^\T(A-C)v_1=-\frac{\beta_1}{\alpha_1}\lambda_1\mbox{ and }v_i^\T(A-C)v_1=0, i>1.\end{equation}

Now, for each $i=2,\dots,n$, let us repeat the construction for $x=v_1+v_i$. From \eqref{dep}, there is some $(\alpha_{i},\beta_{i})\neq(0,0)$ such that:
$${\alpha}_{i}(A-C)v_1+{\alpha}_{i}(A-C)v_i+\beta_{i}(\lambda_1v_1+\lambda_iv_i)=0.$$
Taking the inner product with $v_1$ and using \eqref{aux1} and the fact that $v_1^\T(A-C)v_i=v_i^\T(A-C)v_1$, we have:
$$-{\alpha}_{i}\frac{\beta_1}{\alpha_1}\lambda_1+{\beta}_{i}\lambda_1=0,$$ which implies ${\alpha}_{i}\neq0$ and $\frac{{\beta}_{i}}{{\alpha}_{i}}=\frac{\beta_1}{\alpha_1},i>1$. Thus, we conclude that for $u_1=v_1, u_2=v_1+v_2,\dots,u_n=v_1+v_n$ it holds $$(A-C)u_i+\delta(B-C)u_i=0,\mbox{ for all }i=1,\dots,n,$$ where $\delta=\frac{\beta_1}{\alpha_1}$. Since $\{u_1,\dots,u_n\}$ is a basis of $\R^n$, we have $A-C+\delta(B-C)=0$, and the result follows.
\jo{\qed}
\end{proof}

Note that Lemma $\ref{jacobrank}$ does not hold without the assumption of symmetry of the matrices. A counter-example is $A_1=\left(\begin{array}{cc}1&0\\0&0\end{array}\right)$, $A_2=\left(\begin{array}{cc}1&0\\0&1\end{array}\right)$ and $A_3=\left(\begin{array}{cc}1&0\\1&0\end{array}\right)$ with $a=1$. We give next our result for problem \eqref{quad}.

\begin{theorem}
\label{quadtheo}
If $(x^*,z^*)=(0,0)$ is a local minimizer of \eqref{quad} and the rank of the Jacobian matrix $J(x,z)=[\nabla_{x,z} g_1(x,z) \cdots \nabla_{x,z} g_m(x,z)]$ increases at most by one in a neighborhood with respect to the rank at $(x,z)=(x^*,z^*)$, then there exists $t\in\Lambda_m$ such that $\sum_{i=1}^mt_iA_i$ is positive semidefinite. 
\end{theorem}
\begin{proof}
We have $\nabla_{x,z} g_i(x,z)=\left(\begin{array}{c}A_ix\\-1\end{array}\right), i=1,\dots,m$. Hence, Mangasarian-Fromovitz holds at $(x^*,z^*)$, the critical cone $C(x^*,z^*)$ is the subspace $\R^n\times\{0\}$, and the set of Lagrange multipliers is the simplex $\Lambda_m$. Hence, its vertices are given by $\mu^i=e_i, i=1,\dots,m$ where $e_1,\dots,e_m$ is the canonical basis of $\R^m$ and $\nabla_{x}^2 L(x^*,z^*,\mu^i)=A_i, i=1,\dots,m$.
Since the rank of $J(x^*)=J(x^*,z^*)$ is one, the rank assumption implies that $J(x)$ has rank at most $2$ in some neighborhood of $x^*$. This implies that $J(x)$ has rank at most $2$ for all $x\in\R^n$, and the result follows from Lemma \ref{jacobrank} and Theorem \ref{main}.

Note that the local minimality of $(0,0)$ implies that $\max_{i=1,\dots,m}\{x^\T A_ix\}\geq0$ for all $x$ in a small neighborhood of the origin, which implies it holds for all $x\in\R^n$, hence the result would also follow from Lemma \ref{jacobrank} and Lemma \ref{yuanrank}.
\jo{\qed}
\end{proof}

\section{Conclusions}

Although no-gap necessary and sufficient second-order optimality conditions of Fritz-John type are well established in the optimization literature \cite{bshapiro}, these type of optimality condition has little relations to second-order global convergence results of algorithms. Hence the importance of developing weak conditions ensuring the existence of a Lagrange multiplier with a positive semidefinite Lagrangian Hessian on the critical cone, or a meaningful subset.

We developed a second-order optimality condition of this type where the Lagrange multiplier set can be larger than a bounded line segment, generalizing \cite{baccaritrad}. This was done by an  extension of Yuan's Lemma \cite{yuan} to a set of linearly dependent matrices. The results were applied to address a conjecture from \cite{ams2} about a second-order optimality condition of this type under non-constant rank on a class of quadratically-constrained problems.

\section*{Acknowledgement}
This work was supported by FAPESP (grants 2013/05475-7 and 2016/02092-8) and CNPq. \jo{The author would like to acknowledge the valuable comments and suggestions made by the referees.}


\si{}

\jo{}

\end{document}